\documentclass[UTF8,12pt,scheme=plain,fontset=none]{ctexart}
\usepackage[T1]{fontenc}
\usepackage[utf8]{inputenc}
\usepackage{lmodern}

\usepackage[a4paper,margin=1in]{geometry}
\usepackage{amsmath,amssymb,amsthm,mathtools,amsfonts}
\usepackage{scalerel,stackengine}
\stackMath
\usepackage{enumitem}
\usepackage{microtype}
\usepackage{xcolor}
\usepackage{hyperref}

\hypersetup{
  colorlinks=true,
  linkcolor=blue,
  citecolor=blue,
  urlcolor=blue,
  pdftitle={Nonexistence of Frame Measures for Separated Uniform Consecutive-Digit Bernoulli Convolutions},
  pdfauthor={Xiao-Ye Fu and Wei-Jie Wang},
  pdfsubject={Fourier frames and Bernoulli convolutions},
  pdfkeywords={Fourier frames, frame measures, Bernoulli convolutions, cyclic characters, mask quotients}
}

\numberwithin{equation}{section}

\AtBeginDocument{%
  \setlength{\abovedisplayskip}{6.5pt plus 2pt minus 2pt}%
  \setlength{\belowdisplayskip}{6.5pt plus 2pt minus 2pt}%
  \setlength{\abovedisplayshortskip}{4pt plus 2pt minus 2pt}%
  \setlength{\belowdisplayshortskip}{4pt plus 2pt minus 2pt}%
}

\theoremstyle{plain}
\newtheorem{theorem}{Theorem}[section]

\newtheorem{lemma}[theorem]{Lemma}
\newtheorem{remark}[theorem]{Remark}
\theoremstyle{definition}
\newtheorem{definition}[theorem]{Definition}

\newcommand{\R}{\mathbb R}
\newcommand{\N}{\mathbb N}
\newcommand{\supp}{\operatorname{supp}}

\newcommand{\e}{\mathrm{e}}
\newcommand{\1}{\mathbf{1}}
\newcommand{\norm}[1]{\lVert #1\rVert}
\newcommand{\abs}[1]{\lvert #1\rvert}
\newcommand{\wh}[1]{\widehat{#1}}
\newcommand{\longwh}[1]{%
  \savestack{\tmpbox}{\stretchto{%
    \scaleto{%
      \scalerel*[\widthof{\ensuremath{#1}}]{\kern-.6pt\bigwedge\kern-.6pt}%
      {\rule[-\textheight/2]{1ex}{\textheight}}%
    }{\textheight}%
  }{0.5ex}}%
  \stackon[1pt]{#1}{\tmpbox}%
}
\newcommand{\dd}{\,\mathrm{d}}
\newcommand{\eps}{\varepsilon}

\title{Nonexistence of Frame Measures for Separated Uniform Consecutive-Digit Bernoulli Convolutions}
\author{Xiao-Ye Fu and Wei-Jie Wang}
\date{}

\begin{document}

\maketitle

\begin{abstract}
We investigate the existence problem of frame measures for uniform consecutive-digit Bernoulli convolutions under the separated scaling condition. Given parameters $N\ge2$ and $0<\rho<1/N$, we prove that if $\rho^{-m}=B\in\N$ for some $m\ge1$ and $N\nmid B$, then the associated $N$-Bernoulli convolution \(\mu_{\rho,N}\) admits no frame measure. The proof introduces a new cyclic mask-quotient obstruction adapted to the multi-character structure of the consecutive-digit framework. 
\end{abstract}

\noindent\textbf{Keywords.} Bernoulli convolutions; Fourier frames; frame measures; self-similar measures; cyclic characters; mask quotients.

\medskip
\noindent\textbf{2020 Mathematics Subject Classification.} 42C15, 42C40, 28A80, 42B10.

\section{Introduction}\label{sec:introduction}

Fourier analysis on singular fractal measures centers on a fundamental problem: to what extent does the classical orthogonal exponential decomposition persist when Lebesgue measure is replaced by a singular fractal measure? Spectrality provides the strongest affirmative answer, corresponding to the existence of an orthonormal basis of exponential functions. Fourier frames relax orthogonality while retaining stable reconstruction; we refer to~\cite{Christensen} for general frame and Riesz basis theory. Frame measures unify discrete Fourier frames and their weighted counterparts: they extend discrete frame inequalities by replacing summation over frequencies with integration against a positive Borel measure on the frequency domain. Standard Fourier frames correspond to atomic frame measures induced by counting measures on discrete frequency sets, while weighted Fourier frames correspond to weighted atomic frame measures. 
Meanwhile, the discretization theorem of Freeman and Speegle~\cite{FreemanSpeegle} asserts that bounded continuous frames can be sampled to yield discrete frames. Within the exponential framework studied in this paper, this sampling theorem permits frequency multiplicities.  Precise definitions for frame measures are recalled in Section~\ref{sec:framework}.

\vspace{0.2cm}

Dutkay, Han and Weber~\cite{DutkayHanWeber} introduced frame measures as a tool to study Fourier frame problems for singular self-similar fractal measures. They proved that a compactly supported Borel measure admits a frame measure if and only if it admits an atomic frame measure, which implies the existence of a weighted Fourier frame. Dutkay and Lai~\cite{DutkayLai} subsequently derived a necessary uniformity condition for measures to admit frame measures and this condition can be applied to exclude Fourier frames for non-overlapping self-affine measures with unequal probability weights. A comprehensive survey of Fourier bases and Fourier frames for self-affine measures can be found in~\cite{DutkayLaiWang}. 

\vspace{0.2cm}

A fundamental distinction separates spectrality from the existence of Fourier frames, as demonstrated in~\cite{LaiWang}: a measure may possess only a sparse family of mutually orthogonal exponentials yet still admit a Fourier frame. The construction in~\cite{LaiWang} produces non-spectral Moran-type measures admitting Fourier frames via almost-Parseval-frame towers. These examples, however, do not cover the homogeneous Bernoulli convolutions considered here. On the spectral side, self-affine and Cantor-type measures have been studied through Hadamard structures, affine invariant systems, tree-based frequency constructions, and the arithmetic of consecutive digit sets~\cite{LabaWang,DutkayJorgensen,DutkayHanSun,DaiHeLai,Dai,DaiHeLau}. The existence of Fourier frames for non-spectral homogeneous self-similar
measures has remained largely unresolved for more than two decades, with
decisive progress only recently obtained for two-branch Bernoulli
convolutions.

Among such self-similar measures, the middle-third Cantor measure provides a canonical test case. Jorgensen and Pedersen~\cite{JorgensenPedersen} established its non-spectrality, and Strichartz~\cite{Strichartz} subsequently raised the natural question of whether the weaker Fourier-frame property might still hold. This long-standing problem was resolved in the negative by two independent recent works~\cite{FuSongWang,PontLiehrTaylor}.

%

\vspace{0.2cm}

The negative conclusion for two-branch Bernoulli convolutions relies crucially on the Walsh-quotient obstruction introduced in~\cite{FuSongWang}, which raises a fundamental structural algebraic question:
 is this frame barrier exclusive to the two-element sign group \(\{-1,1\}\), or does it arise universally from character theory of finite cyclic groups of arbitrary order? In the binary setting, there is only one nontrivial character, so a single quotient suffices for the obstruction argument. For \(N\ge3\), however, the scaling \(t\mapsto Bt\) interacts with several nontrivial characters, which requires a family of character-indexed cyclic mask quotients rather than a direct adaptation of the binary quotient. The present work addresses this question uniformly for general \(N\)-digit Bernoulli convolutions with consecutive digits. Fix an integer \(N\ge 2\), contraction ratio \(0<\rho<1\). Let \(A_1,A_2,\dots\) denote independent identically distributed random variables obtained by uniform sampling from \(\{0,1,\dots,N-1\}\). Define
\begin{equation}\label{eq:intro-series}
	 X_{\rho,N}=\sum_{j=1}^{\infty}A_j\rho^j,
	 \qquad \mu_{\rho,N}=\operatorname{Law}(X_{\rho,N}).
	\end{equation}
Equivalently, \(\mu_{\rho,N}\) is the equal-weight self-similar probability measure generated by the affine iterated maps \(S_a(x)=\rho (x+a)\), \(a\in\{0,1,\dots,N-1\}\).

\vspace{0.2cm}

We focus our investigation on the separated parameter regime where cylinder sets do not overlap, which corresponds to the range \(0<\rho<1/N\).
A sharp geometric transition occurs at the critical threshold \(\rho=1/N\). When \(0<\rho<1/N\), the coding map associated with digit sequences is injective, and coordinate characters of the cyclic group \(\mathbb Z/N\mathbb Z\) induce well-defined functions on the support of \(\mu_{\rho,N}\). For \(\rho>1/N\), cylinder sets overlap, and these character packets fail to yield globally consistent functions without further regularization. We therefore restrict our analysis entirely to the separated parameter regime \(0<\rho<1/N\) and do not treat systems with overlapping cylinders in the present paper.

\vspace{0.2cm}

The spectral problem for the \(N\)-digit Bernoulli convolutions \eqref{eq:intro-series} has been completely settled. Comprehensive spectral classifications for such measures were developed in~\cite{DaiHeLai,Dai,DaiHeLau}, showing that the measure admits an orthonormal exponential basis if and only if \(\rho=1/q\) for some integer \(q\) divisible by \(N\). 
 The remaining difficulty is to exclude the more flexible frame structures in nonspectral cases. Our main theorem gives such an obstruction whenever iterated scaling dynamics produce a nonzero residue modulo $N$ after finitely many iterations.

\begin{theorem}\label{thm:main}
Let $N\ge2$ and $0<\rho<1/N$.  Suppose that $\rho^{-m}=B$ for some integers $m\ge1$, $B\ge2$, and that $N\nmid B$.  Then $\mu_{\rho,N}$ admits no frame measure.
\end{theorem}

Since every ordinary Fourier frame induces an atomic frame measure, Theorem~\ref{thm:main} in particular rules out the existence of Fourier frames. The divisibility condition \(N\nmid B\) generalizes the parity obstruction arising in the binary setting \(N=2\). Theorem~\ref{thm:main} recovers the reciprocal-power nonexistence theorem established in~\cite{FuSongWang}. For all integers \(N\ge 3\), it provides new frame obstructions for general consecutive-$N$ digit self-similar measures.

\vspace{0.2cm}

 The proof proceeds by contradiction, combining analytic estimates with a character-theoretic construction based on an abstract contradiction principle (Theorem \ref{thm:abstract}). Assuming a frame measure exists, one derives a uniform positive lower bound
on certain \(L^2\)-norms, while the upper frame bound, a large-value
propagation estimate, and the dominated convergence theorem force the same
quantities to tend to zero. The character-theoretic side constructs unit-norm test functions and cyclic mask quotients using the symbolic coding associated with \(\mu_{\rho,N}\), and the independence inherent to these objects serves to verify exactly the hypotheses required for Theorem \ref{thm:abstract}.

\vspace{0.2cm}

The paper is organized as follows. In Section~\ref{sec:framework}, we compile the preliminary background on frame measures and establish the abstract obstruction principle central to our argument. In Section~\ref{sec:packets}, we construct the cyclic group structure of $N$-digit Bernoulli convolutions through symbolic character packets, twisted masks, and factorizations of Fourier transforms. In Section~\ref{sec:contradiction}, we convert this cyclic algebraic structure into a rigorous obstruction precluding the existence of frame measures and complete the proof of Theorem~\ref{thm:main}.

\section{Preliminaries}\label{sec:framework}

In this section, we introduce the basic theory of frame measures and extract the key analytic principle underlying the proof of Theorem \ref{thm:main}. We adopt the following standard Fourier transform convention throughout the paper: for any finite Borel measure $\sigma$ on $\R$, $\wh\sigma(\lambda)=\int_{\R}\e^{-2\pi i\lambda x}\dd\sigma(x)$.  Similarly, for $f\in L^2(\nu)$,  we define the Fourier transform of the measure \(f\nu\) by 
$$\wh{f\nu}(\lambda)=\int_{\R}f(x)\e^{-2\pi i\lambda x}\dd\nu(x),$$ 
which conforms to the standard frame-measure formalism in~\cite{DutkayHanWeber,DutkayLai}.

\begin{definition}\label{def:frame-measure}
Let $\nu$ denote a compactly supported Borel probability measure on $\R$.
A countable subset $\Lambda\subset\R$ is called a Fourier frame spectrum for $\nu$ if there exist constants $0<A\le C<\infty$ satisfying
\begin{equation}\label{eq:fourier-frame}
 A\norm{f}_{L^2(\nu)}^2
 \le \sum_{\lambda\in\Lambda}\abs{\wh{f\nu}(\lambda)}^2
 \le C\norm{f}_{L^2(\nu)}^2
\end{equation}
for every $f\in L^2(\nu)$.  For the broader setting, a positive Borel measure $\Gamma$ on $\R$ is said to be a frame measure for $\nu$ if
\begin{equation}\label{eq:frame-measure}
 A\norm{f}_{L^2(\nu)}^2
 \le \int_{\R}\abs{\wh{f\nu}(\lambda)}^2\dd\Gamma(\lambda)
 \le C\norm{f}_{L^2(\nu)}^2
\end{equation}
holds for every $f\in L^2(\nu)$ with uniform constants $0<A\le C<\infty$.
\end{definition}

Setting $\Gamma=\sum_{\lambda\in\Lambda}\delta_\lambda$ recovers the classical discrete Fourier frame inequality \eqref{eq:fourier-frame}, whereas $\Gamma=\sum_{\lambda\in\Lambda}w_\lambda\delta_\lambda$ with $w_\lambda>0$ gives a weighted Fourier frame.  Thus, the frame-measure formalism also accommodates genuinely non-atomic frequency measures. Given a frame measure $\Gamma$ associated to $\nu$, we introduce the auxiliary measure $\eta_\Gamma$ via the density relation 
$$\dd\eta_\Gamma(\lambda)=\abs{\wh\nu(\lambda)}^2\dd\Gamma(\lambda).$$ Applying the upper frame inequality to $f\equiv1$ gives $\eta_\Gamma(\R)\le C$, so $\eta_\Gamma$ is a finite Borel measure.

\vspace{0.2cm}

The existence of a frame measure is invariant under translations and nonvanishing dilations. We state this elementary invariance property in the precise form required for subsequent arguments. 

\begin{lemma}\label{lem:affine-invariance}
Let $T(x)=ax+c$ with $a\ne0$, and let $\widetilde\nu=T_\#\nu$, the pushforward measure of \(\nu\) under \(T\).  Then $\nu$ admits a frame measure if and only if $\widetilde\nu$ does.  
\end{lemma}

\begin{proof}
Define the operator $U:L^2(\widetilde\nu)\to L^2(\nu)$ by composition $Uf=f\circ T$. Then this operator is unitary.  For $f\in L^2(\widetilde\nu)$, set $g=Uf$, a straightforward change-of-variable calculation yields
\begin{equation}\label{eq:affine-fourier}
 \wh{f\widetilde\nu}(\lambda)
 =\e^{-2\pi i\lambda c}\wh{g\nu}(a\lambda).
\end{equation}
Suppose that $\Gamma$ is a frame measure for $\nu$, and define $D_{1/a}(\xi)=\xi/a$ and $\widetilde\Gamma=(D_{1/a})_\#\Gamma$.  Then
\begin{align*}
 \int_{\R}\abs{\wh{f\widetilde\nu}(\lambda)}^2\dd\widetilde\Gamma(\lambda)
 &=\int_{\R}\abs{\wh{g\nu}(a\lambda)}^2\dd\widetilde\Gamma(\lambda)\\
 &=\int_{\R}\abs{\wh{g\nu}(\xi)}^2\dd\Gamma(\xi).
\end{align*}
Since $U$ is an isometry on $L^2$-norm, the lower and upper frame bounds are preserved under this transformation.  Repeating the argument for the inverse map $T^{-1}$ establishes the reverse implication. 
\end{proof}


The above analysis does not rely on the Bernoulli convolution structure. We now extract the contradiction mechanism, which we later verify via cyclic packets and associated mask quotients for $N$-Bernoulli convolutions we consider. Our conclusion requires solely a finite-frame-induced measure, pointwise decay bounds, and a uniform propagation estimate, with no self-similarity condition imposed in either the claim or its proof.  

\begin{theorem}\label{thm:abstract}
	Let $\nu$ be a compactly supported Borel probability measure and let \(\Gamma\) be a positive Borel measure. Set
	\[
	\dd\eta_\Gamma(\lambda)=\abs{\wh\nu(\lambda)}^2\dd\Gamma(\lambda).
	\]
	If there exists an integer $s_0$ such that for all integers $s\ge s_0$, we may associate functions $F_s,G_s\in L^2(\nu)$ and measurable functions $q_s,r_s$ on $\R$ satisfying the following conditions.
	\begin{enumerate}[label=\textup{(\roman*)}]
		\item There exist constants $a,b>0$ such that $\norm{F_s}_{L^2(\nu)}^2\ge a$ and $\norm{G_s}_{L^2(\nu)}^2\le b$.
		\item For $\Gamma$-almost every $\lambda$,
		\begin{equation}\label{eq:factor-hyp}
			\abs{\wh{F_s\nu}(\lambda)}^2
			=\abs{\wh\nu(\lambda)}^2\abs{q_s(\lambda)}^2,
		\end{equation}
	 and
		\begin{equation}\label{eq:pair-hyp}
			\abs{\wh\nu(\lambda)}^2\abs{q_s(\lambda)r_s(\lambda)}^2
			\le \abs{\wh{G_s\nu}(\lambda)}^2.
		\end{equation}
		\item For $\eta_\Gamma$-almost every $\lambda$, one has $q_s(\lambda)\to0$ as \(s\to \infty\).
		\item There exist constants $R_0,\kappa>0$ such that, for every $s\ge s_0$ and 	for $\eta_\Gamma$-a.e. $\lambda\in\R$, 
		\begin{equation}\label{eq:prop-hyp}
			\abs{q_s(\lambda)}\ge R_0
			\quad\Longrightarrow\quad
			\abs{r_s(\lambda)}\ge\kappa\abs{q_s(\lambda)}.
		\end{equation}
	\end{enumerate}
	Then \(\Gamma\) cannot be a frame measure for  $\nu$.
\end{theorem}


\begin{proof}
Assume, for contradiction, that $\Gamma$ is a frame measure for $\nu$. Let $A, C$ be its lower and upper frame bounds respectively. Recall the auxiliary measure $\eta=\eta_\Gamma$ constructed earlier, which is a finite Borel measure on \(\mathbb{R}\).

We first establish a uniform lower bound for the integral of \(|q_s|^2\) with respect to \(\eta\).  Combining the definition of $\eta$, \eqref{eq:factor-hyp}, and the lower frame inequality applied to $F_s$, we have
\begin{align}
 \int_{\R}\abs{q_s(\lambda)}^2\dd\eta(\lambda)
 &=\int_{\R}\abs{\wh\nu(\lambda)}^2
   \abs{q_s(\lambda)}^2\dd\Gamma(\lambda)\notag\\
 &=\int_{\R}\abs{\wh{F_s\nu}(\lambda)}^2\dd\Gamma(\lambda)\notag\\
 &\ge A\norm{F_s}_{L^2(\nu)}^2
 \ge Aa.
 \label{eq:lower-moment}
\end{align}

We next derive a uniform upper bound for the integral of \(|q_s r_s|^2\). Using \eqref{eq:pair-hyp} and the upper frame inequality for $G_s$,
\begin{align}
 \int_{\R}\abs{q_s(\lambda)r_s(\lambda)}^2\dd\eta(\lambda)
 &=\int_{\R}\abs{\wh\nu(\lambda)}^2
   \abs{q_s(\lambda)r_s(\lambda)}^2\dd\Gamma(\lambda)\notag\\
 &\le\int_{\R}\abs{\wh{G_s\nu}(\lambda)}^2\dd\Gamma(\lambda)\notag\\
 &\le C\norm{G_s}_{L^2(\nu)}^2
 \le Cb.
 \label{eq:upper-product}
\end{align}

Fix $R\ge R_0$. Decomposing the integral of \(|q_s|^2\) over the level set \(\{|q_s|\le R\}\) and its complement gives 
\begin{equation}\label{eq:energy-split}
 \int_{\R}\abs{q_s}^2\dd\eta
 =\int_{\{\abs{q_s}\le R\}}\abs{q_s}^2\dd\eta
 +\int_{\{\abs{q_s}>R\}}\abs{q_s}^2\dd\eta.
\end{equation}
On the high-level set \(\{|q_s|>R\}\), condition (iv) (the uniform propagation estimate) gives $\abs{q_sr_s}^2\ge\kappa^2\abs{q_s}^4$. Meanwhile, \(|q_s|>R\) implies $\abs{q_s}^2\le R^{-2}\abs{q_s}^4$. Consequently, by \eqref{eq:upper-product},
\begin{align}
 \int_{\{\abs{q_s}>R\}}\abs{q_s}^2\dd\eta
 &\le \frac1{R^2}\int_{\{\abs{q_s}>R\}}\abs{q_s}^4\dd\eta\notag\\
 &\le \frac1{\kappa^2R^2}
 \int_{\{\abs{q_s}>R\}}\abs{q_sr_s}^2\dd\eta\notag\\
 &\le \frac{Cb}{\kappa^2R^2}.
 \label{eq:tail}
\end{align}
This estimate is uniform in $s$. For the bounded low-level component, define the truncated function
 $h_{s,R}(\lambda)=\abs{q_s(\lambda)}^2\1_{\{\abs{q_s(\lambda)}\le R\}}$. Then $0\le h_{s,R}\le R^2$, and hypothesis~(iii) gives $h_{s,R}(\lambda)\to0$ for $\eta$-almost every $\lambda$ as \(s\to\infty\). Since $\eta(\R)<\infty$, the Dominated Convergence Theorem yields
\begin{equation}\label{eq:bounded-part}
 \lim_{s\to\infty}\int_{\R}h_{s,R}\dd\eta=0.
\end{equation}
Taking the \(\limsup\limits_{s\to\infty}\) across both terms in decomposition \eqref{eq:energy-split} and substituting bounds \eqref{eq:tail}--\eqref{eq:bounded-part}, we obtain 
\begin{equation}\label{eq:limsup-energy}
 \limsup_{s\to\infty}\int_{\R}\abs{q_s}^2\dd\eta
 \le\frac{Cb}{\kappa^2R^2},
 \qquad R\ge R_0.
\end{equation}
The right-hand side tends to zero as $R\to\infty$, and therefore $\int_{\R}\abs{q_s}^2\dd\eta\to0$ as \(s\to \infty\). This contradicts the uniform positive lower bound \eqref{eq:lower-moment}. Hence $\Gamma$ is not a frame measure for \(\nu\).
\end{proof}

\section{Cyclic Structure of $\mu_{\rho,N}$}\label{sec:packets}

We now specialize to the consecutive-digit Bernoulli convolution $\mu=\mu_{\rho,N}$ introduced in Section \ref{sec:introduction}, with parameters satisfying $0<\rho<1/N$.  The purpose of this section is to identify  the structural objects that underlie the contradiction mechanism developed in Theorem~\ref{thm:abstract}. 

\subsection{Symbolic coding and coordinate packets}

Let \(\mu_{\rho,N}\) be the equal-weight self-similar probability measure generated by the affine iterated maps \(\{S_a(x)=\rho (x+a)\}_{a=0}^{N-1}\). Let $\Omega=\{0,1,\ldots,N-1\}^{\N}$ denote the one-sided infinite symbolic space over $N$ symbols.  We equip $\Omega$ with the product probability measure $\mathbb P$ under which the coordinate maps are independent and uniformly distributed on $\{0,1,\ldots,N-1\}$.  Define the coding map $\pi:\Omega\to \R$ by
\begin{equation}\label{eq:coding}
 \pi(a_1,a_2,\ldots)=\sum_{j=1}^{\infty}a_j\rho^j.
\end{equation}
Then $\mu_{\rho,N}=\pi_\#\mathbb P$. For the series expansion \eqref{eq:coding}, whenever \(0<\rho<1/N\), the first differing digit dominates the entire remaining tail. This separation property yields the following uniqueness statement for symbolic codings.

\begin{lemma}\label{lem:injective}
If $0<\rho<1/N$, then $\pi$ is injective. Consequently, $\pi$ is a homeomorphism from $\Omega$ onto $\supp\mu_{\rho,N}$, and every finite-coordinate function on $\Omega$ induces a Borel function in $L^2(\mu_{\rho,N})$.
\end{lemma}

\begin{proof}
Let ${\bf{a}}=(a_j)$ and ${\bf{b}}=(b_j)$ be distinct sequences in \(\Omega\), and let $k$ be the first index for which $a_k\ne b_k$. Since $\abs{a_k-b_k}\ge1$, we have
\begin{align}
 \abs{\pi({\bf{a}})-\pi({\bf{b}})}
 &\ge \rho^k-(N-1)\sum_{j>k}\rho^j\notag\\
 &=\rho^k\left(1-\frac{(N-1)\rho}{1-\rho}\right)>0.\label{eq:separation}
\end{align}
The last inequality is obtained since $0<\rho<1/N$. This shows that $\pi$ is injective. Compactness of \(\Omega\) combined with injectivity of \(\pi\) yields the homeomorphism property onto the support of \(\mu_{\rho,N}\). 

For the last claim: take any finite-coordinate function \(\phi:\Omega\to\mathbb{C}\), meaning \(\phi(a_1,a_2,\dots)\) depends only on finitely many coordinates \(a_1,\dots,a_m\) for some fixed \(m\in\mathbb{N}\). Define the pushforward function \(\Phi = \phi\circ\pi^{-1}\) on \(\operatorname{supp}\mu_{\rho,N}\). Since finite-coordinate functions on the compact symbolic space \(\Omega\) are globally bounded continuous functions, \(\Phi\) is bounded Borel measurable on \(\mathbb{R}\). Boundedness immediately gives \(\int_{\mathbb{R}}|\Phi|^2d\mu_{\rho,N}<\infty\), so \(\Phi\in L^2(\mu_{\rho,N})\).
\end{proof}

\begin{remark}
	The homeomorphism \(\pi\) transfers bounded finite-coordinate functions (in particular cylinder indicators) over the symbolic space \(\Omega\) into bounded elements of \(L^2(\mu_{\rho,N})\). This correspondence enables the explicit construction of the test functions \(F_s,G_s\) required in Theorem \ref{thm:abstract} through symbolic combinatorics.
\end{remark}

Let $\omega=\e^{2\pi i/N}$. For $s\in\N$ and $r\in\{1,\ldots,N-1\}$, define \emph{the coordinate character packet \(W_{s,r}:\Omega\to\mathbb C\)} by
\begin{equation}\label{eq:packet}
 W_{s,r}({\bf{a}})=\omega^{ra_s}, \qquad {\bf{a}} = (a_1,a_2,\ldots).
\end{equation}
By Lemma 3.1, this function descends to a well-defined function on \(\pi(\Omega)\subset\mathbb R\). Since each \(W_{s,r}\) depends only on the $s$-th coordinate, it is a finite-coordinate function, and thus induces a bounded element of \(L^2(\mu_{\rho,N})\).

\subsection{Mask factorization and cyclic quotients}

Define the ordinary mask and its character twists respectively by
\begin{equation}\label{eq:masks}
 M(t)=\frac1N\sum_{a=0}^{N-1}\e^{-2\pi iat},
 \qquad
 M_r(t)=\frac1N\sum_{a=0}^{N-1}\omega^{ra}\e^{-2\pi iat},
 \quad 1\le r\le N-1.
\end{equation}

The Fourier transform of the self-similar measure $\mu_{\rho,N}$ possesses a standard infinite product representation associated with the one-step mask.

\begin{lemma}[cf.~\cite{DaiHeLau}]\label{lem:mask-product}
For every $\lambda\in\R$,
\begin{equation}\label{eq:mask-product}
 \wh\mu_{\rho,N}(\lambda)
 =\prod_{j=1}^{\infty}M(\rho^j\lambda).
\end{equation}
Moreover, $\sum_{j\ge1}\abs{1-M(\rho^j\lambda)}<\infty$.  Consequently, the product in \eqref{eq:mask-product} converges absolutely and vanishes if and only if at least one factor $M(\rho^j\lambda) = 0$.
\end{lemma}
 
The ordinary and twisted masks are finite geometric sums, and their ratio defines \(Q_r(t)=M_r(t)/M(t)\). The following lemma describes key pointwise properties of this quotient, including its analyticity and singularity structure.

\begin{lemma}\label{lem:quotient}
	Let the base mask \(M(t)\) and twisted mask \(M_r(t)\) be defined by \eqref{eq:masks}. 
	\begin{enumerate}[label=\textup{(\roman*)},leftmargin=2.2em]
		\item For every \(k \in \{0,1,\dots,N-1\}\), 
		\[M\!\left(\frac{k}{N}\right) = 
		\begin{cases}
			1, & k = 0, \\[4pt]
			0, & k \ne 0,
		\end{cases} \qquad
		M_r\!\left(\frac{k}{N}\right) = 
	\begin{cases}
		1, & k = r, \\[4pt]
		0, & k \neq r.
	\end{cases}\]
\item At any point where $M(t)\neq 0$, the quotient
$Q_r(t):=M_r(t)/M(t)$ is analytic. At each zero
$t=k/N$ with $1\le k\le N-1$, the quotient has a simple pole
if $k=r$, and a removable singularity if $k\neq r$.
We extend $Q_r$ continuously across its removable singularities and set
$Q_r=0$ at its poles.
	\end{enumerate}
\end{lemma}

\begin{proof}
Set $z=\e^{-2\pi it}$. If $z\ne1$, the finite geometric-series formula gives 
$$\sum_{a=0}^{N-1}z^a=(1-z^N)/(1-z).$$
 Likewise, if $\omega^r z\ne1$, then $\sum_{a=0}^{N-1}(\omega^r z)^a=(1-z^N)/(1-\omega^r z)$ since $\omega^{rN}=1$. Hence
\begin{equation*}
 M(t)=\frac1N\frac{1-z^N}{1-z},
 \qquad
 M_r(t)=\frac1N\frac{1-z^N}{1-\omega^r z}.
\end{equation*}
Now $M(t)=0$ exactly when $z^N=1$ and $z\ne1$, that is, when $z=\e^{-2\pi ik/N}$ for some $k\in\{1,\ldots,N-1\}$. At $t=k/N$,
\begin{align*}
 M_r(k/N)
 &=\frac1N\sum_{a=0}^{N-1}
   \e^{2\pi ira/N}\e^{-2\pi ika/N}\\
 &=\frac1N\sum_{a=0}^{N-1}
   \e^{2\pi i(r-k)a/N},
\end{align*}
which equals $1$ if $r=k$ and $0$ otherwise. 

Now consider the quotient \(Q_r(t) = M_r(t)/M(t)\). If \(M(t) \neq 0\), \(Q_r(t)\) is clearly analytic. At \(t = k/N\) with \(1 \le k \le N-1\), we have \(M(k/N) = 0\) but \(M(0) = 1\), so \(M(t)\) is not identically zero. Since \(M(t)\) is trigonometric, its zeros are isolated.
At $t\equiv r/N\pmod1$, the numerator is nonzero by (i). The pole is simple because the derivative of $M(t)$ at that point is nonzero. For \(t=k/N\) with \(k\neq r\), both $M(t)$ and $M_r(t)$ have a simple zero, so the singularity is removable. 
\end{proof}

\subsection{Packet coefficient identities}

The Fourier transform \(\widehat{\mu}(\lambda)\) of the self-similar Bernoulli convolution \(\mu\) is expressed as an infinite product of iterated mask factors $M(\cdot)$ according to Lemma \ref{lem:mask-product}. Acting on \(\mu\), the coordinate character packet \(W_{s,r}\) isolates the base mask \(M(\rho^s\lambda)\) from this product and substitutes it with the twisted mask \(M_r(\rho^s\lambda)\). We now present single-packet and double-packet multiplicative identities for the Fourier transforms of measures obtained by applying these packets.

\begin{lemma}\label{lem:packet-factor}
Let $s\ne u$ and $1\le r,r'\le N-1$. Away from the zero set of the denominator masks defining \(Q_r(\cdot)\) and \(Q_{r'}(\cdot)\), we have
\begin{align}
 \wh{W_{s,r}\mu}(\lambda)
 &=\wh\mu(\lambda)Q_r(\rho^s\lambda),\label{eq:one-packet}\\
 \longwh{W_{s,r}W_{u,r'}\mu}(\lambda)
 &=\wh\mu(\lambda)Q_r(\rho^s\lambda)
 Q_{r'}(\rho^u\lambda).\label{eq:two-packet}
\end{align}
\end{lemma}

\begin{proof}
We first prove the one-packet identity. Fix $n\ge s$ and denote $X_n=\sum_{j=1}^nA_j\rho^j$. Using the symbolic representation of $W_{s,r}$, and the independence of the random coefficients \(\{A_j\}_{j\ge1}\), we have
\begin{align}
 \mathbb E\!\left[\omega^{rA_s}\e^{-2\pi i\lambda X_n}\right]
 &=\mathbb E\!\left[
   \omega^{rA_s}\prod_{j=1}^n
   \e^{-2\pi i\lambda A_j\rho^j}\right]\notag\\
 &=\mathbb E\!\left[
   \omega^{rA_s}\e^{-2\pi i\lambda A_s\rho^s}\right]
   \prod_{\substack{1\le j\le n\\j\ne s}}
   \mathbb E\left[\e^{-2\pi i\lambda A_j\rho^j}\right]\notag\\
 &=M_r(\rho^s\lambda)
   \prod_{\substack{1\le j\le n\\j\ne s}}
   M(\rho^j\lambda).
 \label{eq:finite-one-packet}
\end{align}
As $n\to\infty$, $X_n\to X_{\rho,N}$ almost surely. The random variables on the left of \eqref{eq:finite-one-packet} have modulus one, so the dominated convergence theorem gives
\begin{equation}\label{eq:raw-one-packet}
 \wh{W_{s,r}\mu}(\lambda)
 =M_r(\rho^s\lambda)
  \prod_{j\ne s}M(\rho^j\lambda).
\end{equation}
If $M(\rho^s\lambda)\ne0$, multiplying and dividing the right-hand side by the factor $M(\rho^s\lambda)$, then we obtain
\[\widehat{W_{s,r}\mu}(\lambda)
= \left(\prod_{j=1}^\infty M(\rho^j\lambda)\right)\cdot \frac{M_r(\rho^s\lambda)}{M(\rho^s\lambda)}.\]
 Lemma~\ref{lem:mask-product} then gives $\wh{W_{s,r}\mu}(\lambda)=\wh\mu(\lambda)M_r(\rho^s\lambda)/M(\rho^s\lambda)$, which proves \eqref{eq:one-packet}.

For the two-packet formula, take $n\ge\max\{s,u\}$ to cover both indices \(s,u\) in the partial sum \(X_n\). The same independence factorization gives
\begin{align*}
 &\mathbb E\!\left[
 \omega^{rA_s+r'A_u}\e^{-2\pi i\lambda X_n}\right]\\
 &\quad=M_r(\rho^s\lambda)M_{r'}(\rho^u\lambda)
 \prod_{\substack{1\le j\le n\\j\notin\{s,u\}}}
 M(\rho^j\lambda).
\end{align*}
The dominated convergence theorem and multiplying and dividing by the two nonzero ordinary masks \(M(\rho^s\lambda)M(\rho^u\lambda)\) yield \eqref{eq:two-packet}.
\end{proof}

\section{The Frame-Measure Obstruction}\label{sec:contradiction}

We now turn to the proof of Theorem \ref{thm:main}. In order to apply Theorem \ref{thm:abstract}, we shall verify all its underlying assumptions using the quotient function analysis and coordinate packet identities established in Section~\ref{sec:packets}.


\subsection{Zero-set completion}
The quotient form of our packet identities in Lemma \ref{lem:packet-factor} conceals frequencies at which the Fourier transform \(\widehat{\mu}\) vanishes. Such zero-frequency points cannot be ignored in the single-packet identity, as the frame measure may carry nontrivial mass on these frequencies. The forthcoming lemma extends the single-packet identity to hold on the entire real frequency line, using the globally well-defined  representatives \(Q_r\) constructed in Lemma~\ref{lem:quotient}.

\begin{lemma}\label{lem:zero-completion}
Suppose that $\rho^{-m}=B\in\N$ and $N\nmid B$. Define $r_B\in\{1,\ldots,N-1\}$ by the congruence $r_B\equiv B\pmod N$. Then, for all $s>m$ and  $\lambda\in\R$,
\begin{equation}\label{eq:singleton-completion}
 \abs{\wh{W_{s,1}\mu}(\lambda)}^2
 =\abs{\wh\mu(\lambda)}^2
 \abs{Q_1(\rho^s\lambda)}^2,
\end{equation}
Moreover,
\begin{equation}\label{eq:pair-completion}
 \abs{\wh\mu(\lambda)}^2
 \abs{Q_1(\rho^s\lambda)
 Q_{r_B}(\rho^{s-m}\lambda)}^2
 \le
\abs{\longwh{W_{s,1}W_{s-m,r_B}\mu}(\lambda)}^2.
\end{equation}
\end{lemma}

\begin{proof}
Fix $s>m$ and $\lambda\in\R$. If $\wh\mu(\lambda)\ne0$, Lemma~\ref{lem:mask-product} shows that every ordinary mask factor \(M(\rho^j\lambda)\neq 0\) for all \(j\ge 1\). Both assertions then follow directly from Lemma~\ref{lem:packet-factor}.

 Assume now that $\wh\mu(\lambda)=0$. By the assertion of Lemma~\ref{lem:mask-product}, the zero index set $$\mathcal{Z}(\lambda)=\{j\ge1:M(\rho^j\lambda)=0\}\ne \emptyset.$$  The raw product formula \eqref{eq:raw-one-packet}, with $r=1$, reads
\begin{equation}\label{eq:singleton-product}
 \wh{W_{s,1}\mu}(\lambda)
 =M_1(\rho^s\lambda)
 \prod_{j\ne s}M(\rho^j\lambda).
\end{equation}
If $j\in\mathcal{Z}(\lambda)$ with $j\ne s$, the product on the right of \eqref{eq:singleton-product} contains the zero factor $M(\rho^j\lambda)$, and therefore \( \wh{W_{s,1}\mu}(\lambda)=0\).

The only remaining possibility is $\mathcal{Z}(\lambda)=\{s\}$. In this case, $M(\rho^s\lambda)=0$, and Lemma~\ref{lem:quotient} gives $\rho^s\lambda\equiv k/N\pmod1$ for some $k\in\{1,\ldots,N-1\}$. If $k\ne1$, then Lemma~\ref{lem:quotient}(i) gives $M_1(\rho^s\lambda)=0$, so \eqref{eq:singleton-product} again vanishes. If $k=1$, then
\begin{equation}\label{eq:zero-propagation-calculation}
 \rho^{s-m}\lambda
 =B\rho^s\lambda
 \equiv\frac{B}{N}
 \equiv\frac{r_B}{N}\pmod1.
\end{equation}
Since $N\nmid B$,  Lemma~\ref{lem:quotient} therefore implies $M(\rho^{s-m}\lambda)=0$. The index $s-m$ differs from $s$ since $m\ge1$, which contradicts the assumption $\mathcal{Z}(\lambda)=\{s\}$. In all valid subcases with \(\widehat{\mu}(\lambda)=0\), we obtain \(\widehat{W_{s,1}\mu}(\lambda)=0\). This establishes \eqref{eq:singleton-completion} for all \(\lambda\in\mathbb{R}\). 

For the pair estimate, equality follows from \eqref{eq:two-packet} when $\wh\mu(\lambda)\ne0$. When $\wh\mu(\lambda)=0$, under the global convention for the quotient \(Q_r\) at mask zeros, the product \(Q_1(\rho^s\lambda)Q_{r_B}(\rho^{s-m}\lambda)\) is finite. So the left-hand side of \eqref{eq:pair-completion} is zero, whereas the right-hand side is nonnegative. Hence the inequality holds at every $\lambda\in\R$.
\end{proof}

\subsection{Cyclic pole propagation}

We now establish the large-value propagation property near cyclic poles required for part (iv) of Theorem~\ref{thm:abstract}. The residue \(r_B\) labels the pole location obtained after scaling by integer $B$, and the following lemma establishes a uniform quantitative comparison between the quotient functions in the neighbourhoods of these cyclic poles.

\begin{lemma}\label{lem:pole-propagation}
Assume that $B\in\N$ and $N\nmid B$, and let $r_B\equiv B\pmod N$ with $1\le r_B\le N-1$. There exist constants $R_0,\kappa>0$ such that, whenever $t\not\equiv1/N\pmod1$ and $Bt\not\equiv r_B/N\pmod1$,
\begin{equation}\label{eq:pole-propagation}
 \abs{Q_1(t)}\ge R_0
 \quad\Longrightarrow\quad
 \abs{Q_{r_B}(Bt)}\ge\kappa\abs{Q_1(t)}.
\end{equation}
\end{lemma}

\begin{proof}
Both \(Q_1(t)\) and \(Q_{r_B}(Bt)\) are periodic with period $1$: periodicity of \(Q_1\) holds by construction, while \(Q_{r_B}(B(t+1))=Q_{r_B}(Bt+B)=Q_{r_B}(Bt)\) as \(B\in\mathbb{N}\). It thus suffices to restrict analysis to the fundamental domain \(t\in[0,1]\).

The quotient $Q_1$ possesses exactly one simple pole within \([0, 1]\), located at $1/N$.  Write $t=1/N+\eps$ with $\eps\to0$ and set $\omega=\e^{2\pi i/N}$. The denominator of $Q_1$ expands as
\begin{align*}
 1-\omega\e^{-2\pi it}
 &=1-\e^{2\pi i/N}\e^{-2\pi i(1/N+\eps)}\\
 &=1-\e^{-2\pi i\eps}
 =2\pi i\eps+O(\eps^2),
\end{align*}
whereas the numerator of \(Q_1\) converges to the nonzero constant $1-\e^{-2\pi i/N}$. Consequently, the asymptotic blow-up  in the neighbourhood of the pole satisfies
\begin{equation}\label{eq:q1-asymp}
 \abs{Q_1(t)}
 \sim
 \frac{\abs{1-\e^{-2\pi i/N}}}{2\pi\abs\eps}.
\end{equation}

Next, write $B=qN+r_B$ with $q\in\N\cup\{0\}$ and $1\le r_B\le N-1$. Substituting \(t=1/N+\varepsilon\) into \(Bt\) gives 
$$Bt=q+r_B/N+B\eps.$$ By unit periodicity of $Q_{r_B}$, \(Q_{r_B}(Bt)=Q_{r_B}(r_B/N+B\varepsilon)\), so we are approaching the pole \(r_B/N\) with offset \(B\varepsilon\). The denominator of \(Q_{r_B}\)  expands as 
\begin{align*}
 1-\omega^{r_B}\e^{-2\pi iBt}
 &=1-\e^{2\pi ir_B/N}
       \e^{-2\pi i(q+r_B/N+B\eps)}\\
 &=1-\e^{-2\pi iB\eps}
 =2\pi iB\eps+O(\eps^2),
\end{align*}
while the numerator tends to $1-\e^{-2\pi iB/N}=1-\e^{-2\pi ir_B/N}\ne0$. Therefore
\begin{equation}\label{eq:qB-asymp}
 \abs{Q_{r_B}(Bt)}
 \sim
 \frac{\abs{1-\e^{-2\pi iB/N}}}{2\pi B\abs\eps}.
\end{equation}
Dividing \eqref{eq:qB-asymp} by \eqref{eq:q1-asymp} gives
\begin{equation}\label{eq:ratio-limit}
 \frac{\abs{Q_{r_B}(Bt)}}{\abs{Q_1(t)}}
 \longrightarrow
 L_B:=
 \frac{\abs{1-\e^{-2\pi iB/N}}}
 {B\abs{1-\e^{-2\pi i/N}}}>0.
\end{equation}
Choose $\delta>0$ small enough such that $ \frac{\abs{Q_{r_B}(Bt)}}{\abs{Q_1(t)}}\ge L_B/2$ whenever $0<\abs{t-1/N}<\delta$. On the compact set $K_\delta=[0,1]\setminus(1/N-\delta,1/N+\delta)$, the function $Q_1$ is continuous and therefore bounded. Choose $R_0>\max_{t\in K_\delta}\abs{Q_1(t)}$. If $\abs{Q_1(t)}\ge R_0$, periodicity forces $t$ to lie in the chosen punctured neighborhood of $1/N$ modulo one. In this neighbourhood, the uniform lower bound for the ratio holds, so we conclude \eqref{eq:pole-propagation} by setting \(\kappa=L_B/2\).
\end{proof}

\subsection{Proof of the main theorem}

Based on the preceding preparations, we establish Theorem \ref{thm:main} by checking each condition of Theorem \ref{thm:abstract} one by one.

\begin{proof}[\bf Proof of Theorem~\ref{thm:main}]
We proceed by contradiction. Suppose \(\Gamma\) is a frame measure for \(\mu\) with frame bounds \(A, C>0\). Define the positive Borel measure
 $\dd\eta(\lambda)=\abs{\wh\mu(\lambda)}^2\dd\Gamma(\lambda)$. Then $\eta(\R)\le C$ by applying the upper frame inequality to the constant function.
 
Define the common cyclic pole set by
\begin{equation}\label{eq:common-pole-set}
 \mathcal P=
 \bigcup_{j\ge1}\bigcup_{r=1}^{N-1}
 \left\{\lambda\in\R:
 \rho^j\lambda\equiv\frac rN\pmod1\right\}.
\end{equation}
For every $\lambda\in\mathcal P$, at least one ordinary mask factor \(M(\rho^j\lambda)\) vanishes, which implies that $\mathcal P\subset\{\lambda:\wh\mu(\lambda)=0\}$.  Therefore $\eta(\mathcal P)=0$. To make \(Q_r(t)\) be well-defined over the whole real line including pole points, by Lemma \ref{lem:quotient}, we adopt the following piecewise regularized definition:
\begin{equation}\label{eq:quotient-regularization}
Q_r(t)=
\begin{cases}
	\displaystyle\frac{1-e^{-2\pi i t}}{1-\omega^r e^{-2\pi i t}},& t\not\equiv \dfrac{r}{N}\pmod{1},\\[6pt]
	0,& t\equiv \dfrac{r}{N}\pmod{1}.
\end{cases}
\end{equation}
For $s>m$, set
\begin{equation}\label{eq:abstract-data}
 \begin{aligned}
 F_s=W_{s,1}, \quad G_s=W_{s,1}W_{s-m,r_B}, \quad
 q_s(\lambda)=Q_1(\rho^s\lambda), \quad r_s(\lambda)=Q_{r_B}(\rho^{s-m}\lambda).
 \end{aligned}
\end{equation}
We verify the four hypotheses required for Theorem~\ref{thm:abstract} one by one.

\smallskip
\noindent {\bf {Hypothesis (i).}}
 Note that $\abs{W_{s,r}}=1$ pointwise. Then  \(W_{s,r}\) and the product \(W_{s,r}W_{u,r'}\) satisfy $$\norm{W_{s,r}}_{L^2(\mu_{\rho,N})}=1 \qquad \norm{W_{s,r}W_{u,r'}}_{L^2(\mu_{\rho,N})}=1.$$ 
By construction, $\norm{F_s}_{L^2(\mu)}=\norm{G_s}_{L^2(\mu)}=1$, so hypothesis~(i) holds with $a=b=1$.

\smallskip
\noindent {\bf {Hypothesis (ii).}}
 Lemma~\ref{lem:zero-completion} together with the globally regularized quotient convention \eqref{eq:quotient-regularization}, gives $$\abs{\wh{F_s\mu}(\lambda)}^2=\abs{\wh\mu(\lambda)}^2\abs{q_s(\lambda)}^2, \qquad  \abs{\wh\mu(\lambda)}^2\abs{q_s(\lambda)r_s(\lambda)}^2\le\abs{\wh{G_s\mu}(\lambda)}^2.$$ Hence hypothesis~(ii) holds.

\smallskip
\noindent{\bf{Hypothesis (iii).}}
 For every fixed $\lambda$, one has $\rho^s\lambda\to0$ as \(s\to\infty\).  The explicit formula \eqref{eq:quotient-regularization} gives $Q_1(0)=0$, and $Q_1$ is continuous in the neighborhood of $0$ since $1-\omega\ne0$.  Hence $q_s(\lambda)=Q_1(\rho^s\lambda)\to0$ for every $\lambda\in\R$, and in particular for $\eta$-almost every $\lambda$. This confirms hypothesis~(iii).

\smallskip
\noindent{\bf{Hypothesis (iv).}}
 From the scaling relation $\rho^{-m}=B$, we have $\rho^{s-m}\lambda=B\rho^s\lambda$, and hence $$r_s(\lambda)=Q_{r_B}(B\rho^s\lambda).$$  For $\lambda\notin\mathcal P$, neither \(\rho^s\lambda\) nor \(B\rho^s\lambda\) lies on the pole set of the respective quotient maps, so Lemma~\ref{lem:pole-propagation} supplies constants $R_0,\kappa>0$, independent of $s$ and $\lambda$, such that 
$$\abs{q_s(\lambda)}\ge R_0\Longrightarrow\abs{r_s(\lambda)}\ge\kappa\abs{q_s(\lambda)}.$$
Since $\eta(\mathcal P)=0$, the propagation estimate holds \(\eta\)-almost everywhere, establishing hypothesis~(iv).

All four assumptions of Theorem~\ref{thm:abstract} are now satisfied. Applying Theorem~\ref{thm:abstract} yields a contradiction against the lower frame inequality for the sequence \(\{F_s\}\). Therefore, no frame measure \(\Gamma\) for \(\mu\) can exist.
\end{proof}

\section*{Declaration of Generative AI and AI-Assisted Technologies
in the Manuscript Preparation Process}

During the preparation of this manuscript, generative AI was used for
linguistic polishing of the introduction and for checking the formatting
of the reference list. All authors thoroughly reviewed, revised, and
finalized the manuscript and take full responsibility for the content
of the article.

\bigskip
\noindent\textsc{Xiao-Ye Fu}\\
Hubei Key Laboratory of Mathematical Sciences, College of Mathematics and Statistics,
Central China Normal University, Wuhan, Hubei 430079, China\\
Email address: \texttt{xiaoyefu@ccnu.edu.cn}

\medskip
\noindent\textsc{Wei-Jie Wang}\\
Hubei Key Laboratory of Mathematical Sciences, College of Mathematics and Statistics,
Central China Normal University, Wuhan, Hubei 430079, China\\
Email address: \texttt{wwjmath@163.com}

\end{document}